\documentclass[12pt, twoside, leqno]{article}

\usepackage{amsmath,amsthm}
\usepackage{amssymb}
\usepackage[hidelinks]{hyperref}

\usepackage{enumitem}

\usepackage{graphicx}
\usepackage{tikz-cd}
\usetikzlibrary{shapes,arrows}
\tikzstyle{block} = [rectangle, draw,
text width=7em, text centered, rounded corners,draw=black, minimum height=2.5em,text centered]
\tikzstyle{line} = [draw, -latex', thick]
\tikzset{node distance = 2cm}

\usepackage[T1]{fontenc}

\usepackage{caption}
\newtheorem{theorem}{Theorem}[section]
\newtheorem{corollary}[theorem]{Corollary}
\newtheorem{lemma}[theorem]{Lemma}

\theoremstyle{definition}
\newtheorem{definition}[theorem]{Definition}
\newtheorem{remark}[theorem]{Remark}

\numberwithin{equation}{section}

\begin{document}


\baselineskip=17pt


\title{Representations of Real Numbers by Alternating Perron Series and Their Geometry}

\author{Mykola Moroz\\
Institute of Mathematics\\ 
National Academy of Sciences of Ukraine\\
Tereschenkivska 3\\
01024 Kyiv, Ukraine\\
E-mail: moroznik22@gmail.com}

\date{}

\maketitle


\renewcommand{\thefootnote}{}



\renewcommand{\thefootnote}{\arabic{footnote}}
\setcounter{footnote}{0}


\begin{abstract}
We consider the representation of real numbers by alternating Perron series ($P^-$-representation), which is a generalization of representations of real numbers by Ostrogradsky--Sierpi\'nski--Pierce series (Pierce series),  alternating Sylvester series (second Ostrogradsky series), alternating L\"{u}roth series, etc. Namely, we prove the basic topological and metric properties of $P^-$-representation and find the relationship between $P$-representation and $P^-$-represen\-tation in some measure theory problems.
\end{abstract}

\begin{flushright}
	\textit{Dedicated to my advisor \\ Professor Mykola Pratsiovytyi \\ on his 65th birthday}
\end{flushright}

\section{Introduction}

Let $P$ be any fixed sequence of functions $(\varphi_n)_{n=0}^\infty$ such that $$\varphi_n(x_1,\ldots,x_n)\colon\mathbb{N}^n\rightarrow\mathbb{N}$$ for all $n\in\mathbb{N}$, $\varphi_0=\text{const}\in\mathbb{N}$. It is known {\cite[Theorem 1]{Moroz2024}} that each $x\in(0,1]$ can be uniquely represented as a Perron series defined by the sequence $P$, i.e.,
\begin{gather}\label{PerronSer1}
	x=\sum_{n=0}^{\infty}\frac{r_0 r_1 \cdots r_n}{(p_1-1)p_1(p_2-1)p_2\cdots(p_n-1)p_n p_{n+1}}~,
\end{gather}
where $r_0=\varphi_0$, $r_n=\varphi_n(p_1,\ldots,p_n)$, and $p_n\geq r_{n-1}+1$ for all $n\in\mathbb{N}$. Then series \eqref{PerronSer1} is called the \emph{$P$-representation ($P$-expansion) of $x$}, and is briefly denoted by $\Delta^P_{p_1 p_2 \ldots}$. The number $p_k=p_k(x)$ is called the \emph{$k$th $P$-digit of $x$}. 

For different sequences $P$, corresponding $P$-representations of numbers have both common and distinct topological and metric properties. Some of these properties are described in \cite{Moroz2023,Moroz2024}. In the article \cite{Moroz2023}, the properties of the sequence $g_n=p_n-r_{n-1}$, which hold for almost all $x\in(0,1]$ within specific classes of $P$-representations, are studied.

It is important that the $P$-representation is a generalization of L\"{u}roth expansion \cite{Salat1968,ZhPr2012}, Engel expansion \cite{BGP2023,BP2023,ERS1958,Renyi1962-1}, Sylvester expansion \cite{ERS1958,Galambos1976,Perron,Zhu2014},  Dar\'oczy--K\'atai-Birthday expansion \cite{Galambos1998,Lee2001} (DKB-expansion), etc. Of course, Oppenheim expansion \cite{Galambos1970,Oppenheim1972} is a broader generalization of the representations mentioned above and include the $P$-representation for which $\varphi_0=1$. However, for several reasons described in \cite{Galambos1976,Oppenheim1972}, its special case, known as the Restricted Oppenheim Expansion (ROE), is often considered instead \cite{Galambos1998,Galambos1971,SydorukTorbin2017,WangWu2006}.
Note that ROE is a special case of $P$-representation, namely $P$-representation is ROE if $\varphi_0=1$ and $\varphi_n(x_1,\ldots,x_n)=\varphi_n(x_n)$ for all $n\in\mathbb{N}$. Thus, $P$-representation is an intermediate generalization between restricted Oppenheim expansion and Oppenheim expansion.

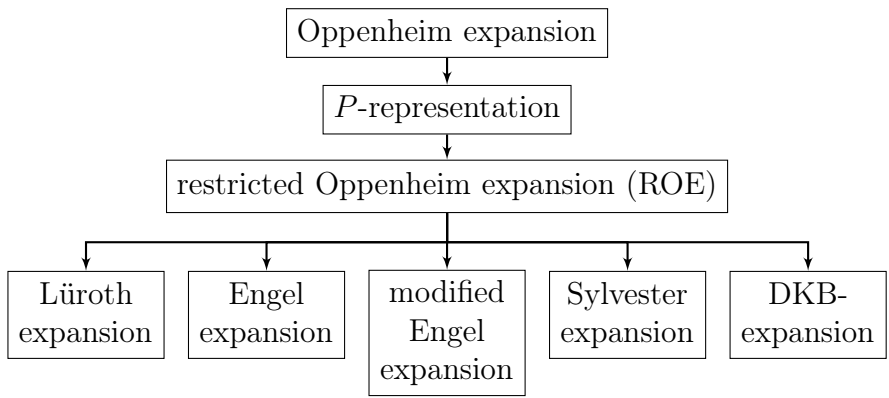
\begin{center}
	\begin{tikzpicture}[scale=3]
		\node[rectangle,draw] (label1) at (0,0) {Oppenheim expansion};
		\node[rectangle,draw] (label2) at (0,-0.34) {$P$-representation};
		\node[rectangle,draw] (label3) at (0,-0.68) {restricted Oppenheim expansion (ROE)};
		\node[rectangle, text width=1.8cm,text centered, draw] (label4) at (-1.6,-1.25) {L\"{u}roth \\ expansion};
		\node[rectangle, text width=1.8cm,text centered, draw] (label5) at (-0.8,-1.25) {Engel \\ expansion};
		\node[rectangle, text width=1.8cm,text centered, draw] (label6) at (0,-1.33) {modified \\ Engel \\ expansion};
		\node[rectangle, text width=1.8cm,text centered, draw] (label7) at (0.8,-1.25) {Sylvester \\ expansion};
		\node[rectangle, text width=1.8cm,text centered, draw] (label8) at (1.6,-1.25) {DKB- \\ expansion};
		
		\draw[line] (label1) -- (label2); 
		\draw[line] (label2) -- (label3); 
		\draw[line] (label3) |-+(0,-0.6em)-| (label4);
		\draw[line] (label3) |-+(0,-0.6em)-| (label5);
		\draw[line] (label3) -- (label6);
		\draw[line] (label3) |-+(0,-0.6em)-| (label7);
		\draw[line] (label3) |-+(0,-0.6em)-| (label8);
	\end{tikzpicture}
\end{center}

The situation with alternating L\"{u}roth expansion, Ost\-rogradsky--Sierpi\'n\-ski--Pierce expansion (Pierce expansion), alternating Sylvester expansion (se\-cond Ostrogradsky expansion) is similar. For a long time, individual theori\-es of these expansions developed independently (for example, see \cite{A2024,ABPT2007,BPT2013,KalpazidouKnopfmacher1991,PB2005,PK2013,Shallit1986,TorbinPratsyovyta2010}). The alternating Oppenheim expansion is a known generalization of these representations (see \cite{GKL2003,Knopfmacher1989}). But, for similar reasons as for the Oppenheim expansion, the alternating analogue of the restricted Oppenheim expansion is predominantly studied \cite{GKL2003}.

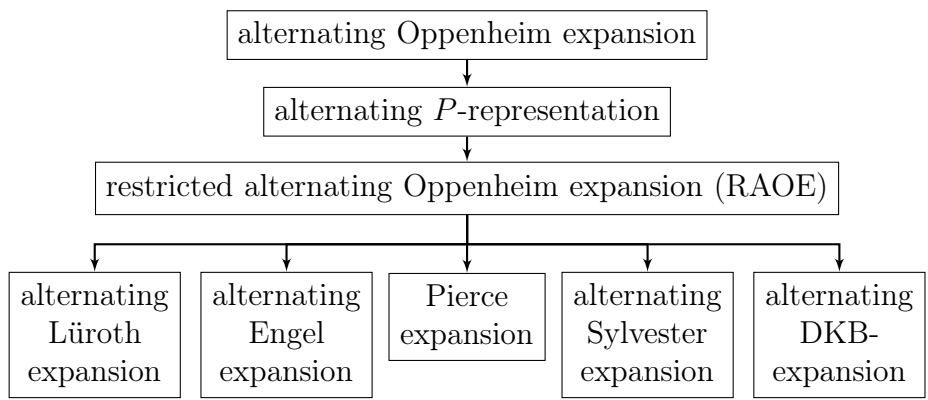
\begin{center}
	\begin{tikzpicture}[scale=3]
		\node[rectangle,draw] (label1) at (0,0) {alternating Oppenheim expansion};
		\node[rectangle,draw] (label2) at (0,-0.34) {alternating $P$-representation};
		\node[rectangle,draw] (label3) at (0,-0.68) {restricted alternating Oppenheim expansion (RAOE)};
		\node[rectangle, text width=2.0cm,text centered, draw] (label4) at (-1.65,-1.33) {alternating \\ L\"{u}roth \\ expansion};
		\node[rectangle, text width=2.0cm,text centered, draw] (label5) at (-0.8,-1.33) {alternating \\ Engel \\ expansion};
		\node[rectangle, text width=1.8cm,text centered, draw] (label6) at (0,-1.25) {Pierce \\ expansion};
		\node[rectangle, text width=2.0cm,text centered, draw] (label7) at (0.8,-1.33) {alternating \\ Sylvester \\ expansion};
		\node[rectangle, text width=2.0cm,text centered, draw] (label8) at (1.65,-1.33) {alternating \\ DKB-expansion};
		
		\draw[line] (label1) -- (label2); 
		\draw[line] (label2) -- (label3); 
		\draw[line] (label3) |-+(0,-0.6em)-| (label4);
		\draw[line] (label3) |-+(0,-0.6em)-| (label5);
		\draw[line] (label3) -- (label6);
		\draw[line] (label3) |-+(0,-0.6em)-| (label7);
		\draw[line] (label3) |-+(0,-0.6em)-| (label8);
	\end{tikzpicture}
\end{center}

In this paper, we consider the $P^-$-representation of real numbers, which is the alternating analogue of the $P$-representation. It generalizes the aforementioned alternating expansions and the restricted alternating Oppenheim expansion. We establish and prove basic topological and metric properties of $P^-$-representation. To achieve this, we adopt a geometric approach similar to that in \cite{Moroz2024}. This approach differs from the traditional method, which relies on algorithms for series expansion (see \cite{Galambos1976}, etc.). We also show that the corresponding $P^-$-representation and $P$-representation have equivalent metric theories, which is the cause of the identity  of some facts for these representations.

In fact, we show that the Lebesgue measure problems related to $P^-$-representation do not require individual solutions, but are instead simple corollaries of the corresponding problems defined in terms of $P$-representa\-tion. In particular, it explains why certain propositions are identical for well-known pairs of representations of real numbers: positive and alternating L\"{u}roth expansions, modified Engel and Pierce expansions, Sylvester and second Ostrogradsky expansions. Relevant examples are given in the last section of this paper.

\section{Definitions and examples}

\begin{definition}
The \emph{alternating Perron series} is defined as a numerical series of the form
	\begin{gather}\label{alternatingPerronSer1}
		\sum_{n=0}^{\infty}\frac{(-1)^n r_0 r_1 \cdots r_n}{(q_1-1)q_1(q_2-1)q_2\cdots (q_n-1)q_n (q_{n+1}-1)},
	\end{gather}
where $(r_n)_{n=0}^{\infty}$ is an arbitrary sequence of natural numbers, $q_n\in\mathbb{N}$, $q_{n}\geq r_{n-1}+1$ for all $n\in\mathbb{N}$.
\end{definition}	

\begin{remark}
It is obvious that $$\frac{r_0 r_1 \cdots r_n}{(q_1-1)q_1(q_2-1)q_2\cdots (q_n-1)q_n (q_{n+1}-1)}\overset{\mathrm{def}}{=}\frac{r_0 }{q_{1}-1}$$ if $n=0$.
\end{remark}

\begin{lemma}\label{lemmaalternatingPerron1}
The alternating Perron series \eqref{alternatingPerronSer1} is absolutely convergent, and its sum is a number from $(0,1)$.
\end{lemma}

\begin{proof}
	Series \eqref{alternatingPerronSer1} is absolutely convergent because
	\begin{align*}
		0&<\sum_{n=0}^{\infty}\frac{r_0 \cdots r_n}{(q_1-1)q_1\cdots (q_n-1)q_n (q_{n+1}-1)}\\
		&\leq \sum_{n=0}^{\infty}\frac{r_0 \cdots r_n}{r_0(r_0+1) \cdots r_{n-1}(r_{n-1}+1)r_n} \\
		&=1+\sum_{n=1}^{\infty}\frac{1}{(r_0+1)\cdots(r_{n-1}+1)}\leq 1+\sum_{n=1}^{\infty}\frac{1}{2^{n}}=2.
	\end{align*}
	Since $\frac{r_0 \cdots r_{n-1}}{(q_1-1)q_1\cdots (q_{n-1}-1)q_{n-1}(q_{n}-1)}>\frac{r_0 \cdots r_n}{(q_1-1)q_1\cdots(q_n-1)q_n(q_{n+1}-1)}$ for all $n\in\mathbb{N}$, we see that
	$$0<\sum_{n=0}^{\infty}\frac{(-1)^n r_0 r_1 \cdot\ldots\cdot r_n}{(q_1-1)q_1(q_2-1)q_2\cdots(q_n-1)q_n (q_{n+1}-1)}<\frac{r_0}{q_1-1}\leq 1.$$
\end{proof}	

Series \eqref{alternatingPerronSer1} is: 
\begin{enumerate}[label=\upshape(\roman*), leftmargin=*, widest=iii]
	\item the alternating L\"{u}roth series when $r_n=1$ for all $n\in\mathbb{N}\cup\{0\}$:
	\begin{multline*}
\frac{1}{q_1-1}-\frac{1}{(q_1-1)q_1(q_2-1)}+\cdots\\	+\frac{(-1)^{n+1}}{(q_1-1)q_1\cdots(q_{n-1}-1)q_{n-1}(q_{n}-1)}+\cdots,
	\end{multline*}
	
	where $q_{n}\geq 2$;
	\item the alternating Engel series when $r_0=1$ and $r_n=q_n-1$ for all $n\in\mathbb{N}$:
	$$\frac{1}{q_1-1}-\frac{1}{q_1(q_2-1)}+\cdots+\frac{(-1)^{n+1}}{q_1\cdots q_{n-1}(q_n-1)}+\cdots,$$
	where $q_1\geq 2$ and $q_{n+1}\geq q_n$;
	\item the Pierce series (the Ostrogradsky--Sierpi\'nski--Pierce series) when $r_0=1$ and $r_n=q_n$ for all $n\in\mathbb{N}$:
	$$\frac{1}{q_1-1}-\frac{1}{(q_1-1)(q_2-1)}+\cdots+\frac{(-1)^{n+1}}{(q_1-1)\cdots(q_n-1)}+\cdots,$$
	where $q_1\geq 2$ and $q_{n+1}\geq q_n+1$;
	\item the alternating Sylvester series (the second Ostrogradsky series) when $r_0=1$ and $r_n=(q_n-1)q_n$ for all $n\in\mathbb{N}$:
	$$\frac{1}{q_1-1}-\frac{1}{q_2-1}+\cdots+\frac{(-1)^{n+1}}{q_{n}-1}+\cdots,$$
	where $q_{1}\geq 2$ and $q_{n+1}\geq (q_n-1)q_n+1$
.
\end{enumerate}

We see that the proposed alternating Perron series generate versions of well-known alternating series that slightly differ from the generally accepted forms. By utilizing this notation of alternating Perron series, we can establish significant relationships between alternating Perron expansions and their positive analogues.

\section{The $P^-$-representation of real numbers. Cylindrical sets}

Let us consider functions $\varphi_n(x_1,\ldots,x_n)\colon\mathbb{N}^ n\rightarrow\mathbb{N}$ for all $n\in\mathbb{N}$ and $\varphi_0=\text{const}\in\mathbb{N}$. We denote a fixed sequence of functions $(\varphi_n)_{n=0}^\infty$ by $P$.

\begin{definition}
Let a number $x\in(0,1)$ be the sum of series \eqref{alternatingPerronSer1}, $r_0=\varphi_0$, and $r_n=\varphi_n(q_1,\ldots,q_n)$ for all $n\in\mathbb{N}$. Then we call series expansion \eqref{alternatingPerronSer1} of $x$ the \emph{$P^-$-representation of $x$}.
\end{definition}

The number $x$ and its $P^-$-representation \eqref{alternatingPerronSer1} are briefly denoted by $\Delta^{P^-}_{q_1 q_2\ldots}$. The number $q_n=q_n(x)$ is called the \emph{$n$th $P^-$-digit of $x$}.

\begin{definition}
A non-empty set $$\Delta^{P^-}_{c_1 \ldots c_k}=\left\{x\in(0,1)\colon  x=\Delta^{P^-}_{c_1 \ldots c_k q_{k+1}q_{k+2}\ldots}\right\}$$ is called a \emph{$P^-$-cylinder of rank $k$ with base $c_1 \ldots c_k$}.
\end{definition}

That is, $\Delta^{P^-}_{c_1 \ldots c_k}$ is the set of all numbers $x\in(0,1)$ that have a $P^-$-repre\-sentation with first $k$ $P^-$-digits $c_1,\ldots,c_k$ respectively. If such a set is empty for some base $c_1 \ldots c_k$, then we do not call this set $P^-$-cylinder.

\begin{remark}
Sets similar to $P^-$-cylinders of rank $k$ are sometimes called \emph{$k$th order cylinders} \cite{Fang2015,Galambos1976} or \emph{$k$th level fundamental intervals} \cite{A2024}. We follow the tradition established in the works of Professor Pratsiovytyi and his colleagues (for example, see \cite{ABPT2007,BP2023,BPT2007,Moroz2024,Moroz2023,PK2013,SydorukTorbin2017,TorbinPratsyovyta2010,ZhPr2012}).
\end{remark}

From the definition of $P^-$-cylinder, it follows that
$\Delta^{P^-}_{c_1 \ldots c_k c_{k+1}}\subset \Delta^{P^-}_{c_1 \ldots c_k}$.

\begin{lemma}\label{sup1alternatingPerron}
	For the $P^-$-cylinder $\Delta^{P^-}_{c_1 \ldots c_k}$ of odd rank $k$, we have
	\begin{gather}
		\sup \Delta^{P^-}_{c_1\ldots c_k}=\sum_{n=0}^{k-1}\frac{(-1)^n r_0 \cdots r_{n}}{(c_1-1)c_1\cdots (c_n-1)c_n (c_{n+1}-1)}, \label{sup1alternatingPerron2}\\
		\inf \Delta^{P^-}_{c_1 \ldots c_k}=\sup \Delta^{P^-}_{c_1 \cdots c_k}-\frac{r_0\cdots r_{k-1}}{(c_1-1)c_1\cdots (c_{k}-1)c_k},\label{inf1alternatingPerron2}
	\end{gather}
	where $r_0=\varphi_0$ and $r_n=\varphi_n(c_1,\ldots,c_n)$ for all $n=1,\ldots,k-1$.
\end{lemma}

\begin{proof}
First, let us prove equality \eqref{sup1alternatingPerron2} for $k=1$. Let $\Delta^{P^-}_{c_1}$ be a $P^-$-cylinder of rank $1$.
Number $\frac{r_0}{c_1-1}$ is an upper bound of $\Delta^{P^-}_{c_1}$. Let us prove that $P^-$-cylinder $\Delta^{P^-}_{c_1}$ contains numbers that are arbitrarily close to $\frac{r_0}{c_1-1}$. 
	
Consider series \eqref{alternatingPerronSer1} in which sequences $(r_n)_{n=0}^\infty$ and $(q_n)_{n=1}^\infty$ are defined by the recurrence relation: 
	$$r_0=\varphi_0,~~q_1=c_1,~~r_n=\varphi_n(q_1,\ldots,q_n),~~q_{n+1}=r_{n}d+1$$
for all $n\in\mathbb{N}$, where $d\in\mathbb{N}$, $d\geq2$. Denote the sum of this series by $x_d$. It is clear that $x_d\in\Delta^{P^-}_{c_1}$, thus  $x_d<\frac{r_0}{c_1-1}$ and 
	\begin{align*}
		x_d&=\frac{r_0}{c_1-1}+\sum_{n=1}^{\infty}\frac{(-1)^n r_0 r_1 \cdots r_n}{(c_1-1)c_1\cdot r_1 d\cdot(r_1d+1)\cdot\ldots\cdot r_{n-1} d\cdot(r_{n-1}d+1)r_n d}\\
		&>\frac{r_0}{c_1-1}-\frac{r_0}{(c_1-1)c_1}\cdot\sum_{n=1}^{\infty}\frac{r_1 \cdots r_n}{r_1^2 d^2  \cdots r_{n-1}^2 d^2 \cdot r_n d} \\
		&=\frac{r_0}{c_1-1}-\frac{r_0}{(c_1-1)c_1}\cdot\sum_{n=1}^{\infty}\frac{1}{r_1 \cdots r_{n-1} d^{2n-1}}\\
		&\geq\frac{r_0}{c_1-1}-\frac{r_0}{(c_1-1)c_1}\cdot\sum_{n=1}^{\infty}\frac{1}{d^{2n-1}}\\
		&=\frac{r_0}{c_1-1}-\frac{r_0}{(c_1-1)c_1}\cdot\frac{p}{d^2-1}>\frac{r_0}{c_1-1}-\frac{r_0}{(c_1-1)c_1}\cdot\frac{1}{d-1}.
	\end{align*}
Since $\frac{r_0}{c_1-1}-\frac{r_0}{(c_1-1)c_1}\cdot\frac{1}{d-1}\to\frac{r_0}{c_1-1}$ as $d\to\infty$, it follows that $x_d\to\frac{r_0}{c_1-1}$ as $d\to\infty$. Therefore, every neighbourhood of $\frac{r_0}{c_1-1}$ contains a point of the cylinder $\Delta^{P^-}_{c_1}$. Thus, $\sup\Delta^{P^-}_{c_1}=\frac{r_0}{c_1-1}$. For $k=1$, equality \eqref{sup1alternatingPerron2} is proved.
	
Equality \eqref{sup1alternatingPerron2} for odd $k\geq 3$ and equality \eqref{inf1alternatingPerron2} are proved similarly.
\end{proof}

\begin{lemma}\label{infsup2alternatingPerron}
For the $P^-$-cylinder $\Delta^{P^-}_{c_1 \ldots c_k}$ of even rank $k$, we have
	\begin{gather}
		\inf \Delta^{P^-}_{c_1 \ldots c_k}=\sum_{n=0}^{k-1}\frac{(-1)^n r_0 \cdots r_{n}}{(c_1-1)c_1\cdots (c_n-1)c_n (c_{n+1}-1)}, \label{inf2alternatingPerron2}\\		
		\sup \Delta^{P^-}_{c_1 \ldots c_k}=\inf \Delta^{P^-}_{c_1 \ldots c_k}+\frac{r_0\cdots r_{k-1}}{(c_1-1)c_1\cdots (c_{k}-1)c_k},\label{sup2alternatingPerron2}
	\end{gather}
	where $r_0=\varphi_0$ and $r_n=\varphi_n(c_1,\ldots,c_n)$ for all $n=1,\ldots,k-1$.
\end{lemma}

The proof of Lemma \ref{infsup2alternatingPerron} is analogous to the proof of Lemma \ref{sup1alternatingPerron}.

\begin{remark}
From the proofs of Lemma \ref{sup1alternatingPerron} and Lemma \ref{infsup2alternatingPerron}, it follows that the infimum and supremum of a $P^-$-cylinder do not belong to this $P^-$-cylinder.
\end{remark}

Henceforth, for each bounded set $A\subseteq\mathbb{R}$, value $\sup A-\inf A$ will be denoted by $\left|A\right|$.

\begin{corollary}
	\begin{gather}
		|\Delta^{P^-}_{c_1 \ldots c_k}|=\frac{r_0\cdots r_{k-1}}{(c_1-1)c_1\cdots (c_{k}-1)c_k},\label{cylalternatingPerron2}	
	\end{gather}
	where $r_0=\varphi_0$ and $r_n=\varphi_n(c_1,\ldots,c_n)$ for all $n=1,\ldots,k-1$.
\end{corollary}

\begin{corollary}
For the $P^-$-cylinder $\Delta^{P^-}_{c_1 \ldots c_k c_{k+1}}$ of even rank $k+1$, we have
	\begin{gather}
		\inf \Delta^{P^-}_{c_1 \ldots c_{k+1}}=\sup \Delta^{P^-}_{c_1 \ldots c_{k}}-\frac{r_0\cdots r_{k}}{(c_1-1)c_1\cdots (c_{k}-1)c_k(c_{k+1}-1)},\label{inf1alternatingPerron3}\\
		\sup \Delta^{P^-}_{c_1\ldots c_{k+1}}=\sup \Delta^{P^-}_{c_1 \ldots c_{k}}-\frac{r_0\cdots r_{k}}{(c_1-1)c_1\cdots (c_{k}-1)c_k c_{k+1}},\label{sup1alternatingPerron3}
	\end{gather}
	where $r_0=\varphi_0$ and $r_n=\varphi_n(c_1,\ldots,c_n)$ for all $n=1,\ldots,k$.
\end{corollary}

\begin{corollary}
For the $P^-$-cylinder $\Delta^{P^-}_{c_1 \ldots c_k c_{k+1}}$ of odd rank $k+1$, we have
	\begin{gather}
		\sup \Delta^{P^-}_{c_1 \ldots c_{k+1}}=\inf \Delta^{P^-}_{c_1 \ldots c_{k}}+\frac{r_0\cdots r_{k}}{(c_1-1)c_1\cdots (c_{k}-1)c_k(c_{k+1}-1)},\label{sup2alternatingPerron3}\\
		\inf \Delta^{P^-}_{c_1\ldots c_{k+1}}=\inf \Delta^{P^-}_{c_1 \ldots c_{k}}+\frac{r_0\cdots r_{k}}{(c_1-1)c_1\cdots (c_{k}-1)c_k c_{k+1}},\label{inf2alternatingPerron3}
	\end{gather}
	where $r_0=\varphi_0$ and $r_n=\varphi_n(c_1,\ldots,c_n)$ for all $n=1,\ldots,k$.
\end{corollary}

\begin{corollary}\label{corollaryinfsupalt}
	For the $P^-$-cylinder $\Delta^{P^-}_{c_1 \ldots c_{k-1} c_k}$ of rank $k$,
	\begin{enumerate}[label=\upshape(\roman*), leftmargin=*, widest=ii]
		\item $\inf\Delta^{P^-}_{c_1 \ldots c_{k-1} c_k}=\sup\Delta^{P^-}_{c_1 \ldots c_{k-1} \left[c_k+1\right]}$ if $k$ is odd;
		\item $\sup\Delta^{P^-}_{c_1 \ldots c_{k-1} c_k}=\inf\Delta^{P^-}_{c_1 \ldots c_{k-1} \left[c_k+1\right]}$ if $k$ is even.
	\end{enumerate}
\end{corollary}

If $x\in(0,1)$ is the infimum or supremum of a $P^-$-cylinder of rank $k$, then $x$ does not belong to any $P^-$-cylinder and is neither the infimum nor the supremum of any $P^-$-cylinder of rank $t$ for $t>k+1$. This means that corresponding $P^-$-representations do not exist for the infimum and supremum of $P^-$-cylinders. We denote the set of infima and suprema of all $P^-$-cylinders by $IS^{P^-}$. Obviously, the set $IS^{P^-}$ is countably infinite.

We will prove that the $P^-$-representation exists and is unique for all $x\in(0,1)\setminus IS^{P^-}$.

\begin{lemma}\label{lemmacylalternatingPerron}
Let $P$ be any fixed sequence of functions $\varphi_i$. For all $n\in\mathbb{N}$ and all $x\in(0,1)\setminus IS^{P^-}$, there exists a unique $P^-$-cylinder $\Delta^{P^-}_{q_1 \ldots q_n}$ of rank $n$ such that
	\begin{align}\label{magoralternatingPerrontheorem}
		\inf \Delta^{P^-}_{q_1 \ldots q_n}<x<\sup\Delta^{P^-}_{q_1 \ldots q_n}.
	\end{align}
\end{lemma}

\begin{proof}
Consider an arbitrary real number $x\in(0,1)\setminus IS^{P^-}$. We give a proof by induction on $n$. 
	
\emph{Base case:} We will show that the proposition holds for $n=1$. 
In view of equalities \eqref{sup1alternatingPerron2} and \eqref{inf1alternatingPerron2}, the inequality $\inf \Delta^{P^-}_{q_1}<x<\sup\Delta^{P^-}_{q_1}$ is equivalent to the inequality
	\begin{align}\label{1}
		q_1>\frac{r_0}{x}>q_1-1,~~r_0=\varphi_0.
	\end{align}
Since $\frac{r_0}{x}>r_0$, there exists a unique natural number $q_1\geq r_0+1$ such that the condition \eqref{1} is satisfied. We denote this
number by $q'_1$. Then $\Delta^{P^-}_{q'_1}$ is the unique $P^-$-cylinder of rank $1$ that satisfies inequality \eqref{magoralternatingPerrontheorem}. Note that inequality \eqref{1} is strict; otherwise, $x\in IS^{P^-}$.
	
We can write that $x=\sup\Delta^{P^-}_{q'_1}-x_1=\frac{r_0}{q'_1-1}-x_1$. Then 
	\begin{align}\label{2}
		0<x_1<\sup\Delta^{P^-}_{q'_1}-\inf\Delta^{P^-}_{q'_1}=\frac{r_0}{(q'_1-1)q'_1}.
	\end{align}
	
We will show that the statement holds for $n=2$. Suppose that $\Delta^{P^-}_{q_1 q_2}$ is $P^-$-cylinder of rank $2$ such that $\inf \Delta^{P^-}_{q_1 q_2}<x<\sup\Delta^{P^-}_{q_1 q_2}$. Then $\inf \Delta^{P^-}_{q_1}<x<\sup\Delta^{P^-}_{q_1}$ because $\Delta^{P^-}_{q_1 q_2}\subset\Delta^{P^-}_{q_1}$. Thus, $q_1=q'_1$, $r_1=\varphi_1(q'_1)$.
	
The inequality $\inf\Delta^{P^-}_{q'_1 q_2}<x<\sup\Delta^{P^-}_{q'_1 q_2}$ is equivalent to inequalities
	\begin{gather}
		\frac{r_0 r_1}{(q'_1-1)q'_1q_2}<x_1<\frac{r_0 r_1}{(q'_1-1)q'_1(q_2-1)}, \notag\\
		q_2>\frac{r_0 r_1}{(q'_1-1)q'_1x_1}>q_2-1~.\label{4}
	\end{gather}
Note that inequality \eqref{4} is strict; otherwise, $x\in IS^{P^-}$.
	
It follows from \eqref{2} that $\frac{r_0 r_1}{(q'_1-1)q'_1 x_1}>r_1$. Therefore, there exists a unique natural number $q_2\geq r_1+1$ that satisfies inequality \eqref{4}. We denote this number by $q'_2$. Then $\Delta^{P^-}_{q'_1 q'_2}$ is the unique $P^-$-cylinder of rank $2$ that satisfies inequality \eqref{magoralternatingPerrontheorem}.
	
\emph{Induction hypothesis:} For $n=k$, assume that there exists a unique $P^-$-cylinder $\Delta^{P^-}_{q'_1 \ldots q'_k}$ of rank $k$ that
	\begin{gather}\label{3}
		\inf\Delta^{P^-}_{q'_1 \ldots q'_k}<x<\sup\Delta^{P^-}_{q'_1 \ldots q'_k}
	\end{gather}
and
	\begin{gather}\label{5}
		0<x_k<\sup\Delta^{P^-}_{q'_1 \ldots q'_k}-\inf\Delta^{P^-}_{q'_1  \ldots q'_k}=\frac{r_0 r_1\cdots r_{k-1}}{(q'_1-1)q'_1\cdots (q'_k-1)q'_k},
	\end{gather}
	where 
	$$\left[ 
	\begin{aligned} 
		&x_k=x-\inf\Delta^{P^-}_{q'_1\ldots q'_k}&&\text{ if }k\text{ is even}, \\ 
		&x_k=\sup\Delta^{P^-}_{q'_1\ldots q'_k}-x&&\text{ if }k\text{ is odd},\\ 
	\end{aligned} 
	\right.$$
	$r_0=\varphi_0$ and $r_i=\varphi_i(q'_1,\ldots,q'_i)$ for all $i=1,\ldots,k-1$.
	
\emph{Induction step:} We will show that the proposition holds for $n=k+1$. If there exists a $P^-$-cylinder $\Delta^{P^-}_{q_1 \ldots q_k q_{k+1}}$ of rank $k+1$ such that $\inf\Delta^{P^-}_{q_1 \ldots q_k q_{k+1}}<x<\sup\Delta^{P^-}_{q_1 \ldots q_k q_{k+1}}$, then $q_1=q'_1,~\ldots,~q_k=q'_k$, $r_0=\varphi_0$, and $r_i=\varphi_i(q'_1,\ldots,q'_i)$ for all $i=1,\ldots,k$. Therefore, we will show that there exists a unique natural number $q_{k+1}\geq r_k+1$ such that $\inf\Delta^{P^-}_{q'_1 \ldots q'_k q_{k+1}}<x<\sup\Delta^{P^-}_{q'_1 \ldots q'_k q_{k+1}}$.
	
\emph{1. Let $k$ be an even.} It follows from the induction hypothesis (equalities \eqref{3} and \eqref{5}) and equalities \eqref{sup2alternatingPerron3} and \eqref{inf2alternatingPerron3} that the inequality $\inf\Delta^{P^-}_{q'_1 \ldots q'_k q_{k+1}}<x<\sup\Delta^{P^-}_{q'_1 \ldots q'_k q_{k+1}}$ is  equivalent to inequalities
	\begin{gather}
		\frac{r_0 r_1\cdots r_k}{(q'_1-1)q'_1\cdots (q'_k-1)q'_k q_{k+1}} <x_k<\frac{r_0 r_1\cdots r_k}{(q'_1-1)q'_1\cdots (q'_k-1)q'_k (q_{k+1}-1)};\notag\\
		q_{k+1}>\frac{r_0 r_1\cdots r_k}{(q'_1-1)q'_1\cdots (q'_k-1)q'_k x_k}>q_{k+1}-1.\label{6}
	\end{gather}
Similarly, inequality \eqref{6} is strict; otherwise, $x\in IS^{P^-}$.
	
It follows from equality \eqref{5} that $\frac{r_0 r_1\cdots r_k}{(q'_1-1)q'_1\cdots (q'_k-1)q'_k x_k}> r_k$. Therefore, there exists a unique natural number $q_{k+1}\geq r_k+1$ that satisfies inequality \eqref{6}. We denote this number by $q'_{k+1}$. Then $\Delta^{P^-}_{q'_1 \ldots q'_k q'_{k+1}}$ is the unique $P^-$-cylinder of rank $k+1$ that satisfies inequality \eqref{magoralternatingPerrontheorem}.
	
\emph{2. Let $k$ be an odd.} In this case, the proof is similar to that of the induction step for even $k$.

By mathematical induction, it follows that the statement holds for all $n\in\mathbb{N}$.
\end{proof}

\begin{lemma}\label{limcylalternatingPerron}
For each sequence $(q_n)_{n=1}^{\infty}$ such that $q_{n}\geq r_{n-1}+1$, with $r_{n-1}=\varphi_{n-1}(q_1,\ldots,q_{n-1})$, we have
$$|\Delta^{P^-}_{q_1\ldots q_n}|\to 0\text{ as }n\to\infty.$$
\end{lemma}

\begin{proof}
Let the sequence $(q_n)_{n=1}^{\infty}$ satisfy the conditions of this Lemma. Then, according to equality \eqref{cylalternatingPerron2}, we have
	\begin{align*}
		0&<|\Delta^{P^-}_{q_1\ldots q_n}|=\frac{r_0\cdots r_{n-1}}{(q_1-1)q_1\cdots (q_n-1)q_n}\\
		&\leq\frac{r_0 \cdots r_{n-1}}{r_0 (r_0+1)\cdots r_{n-1}(r_{n-1}+1)}=\frac{1}{(r_0+1)\cdots(r_{n-1}+1)}\leq\frac{1}{2^n}\to 0.
	\end{align*}
Therefore, $|\Delta^{P^-}_{q_1\ldots q_n}|\to 0\text{ as }n\to\infty$.
\end{proof}

\begin{theorem}\label{mainalternatingtheorem}
	For each sequence $P$ of functions $\varphi_n$, every $x\in(0,1)\setminus IS^{P^-}$ has a unique $P^-$-representation, i.e., there exists a unique sequence of natural numbers $(q_n)_{n=1}^\infty$ such that
	$$x=\sum_{n=0}^{\infty}\frac{(-1)^n r_0 r_1 \cdots r_n}{(q_1-1)q_1\cdots (q_n-1)q_n (q_{n+1}-1)}\equiv\Delta^{P^-}_{q_1 q_2 \ldots},$$
	where $r_0=\varphi_0$, $r_n=\varphi_n(q_1,\ldots,q_n)$, and  $q_{n}\geq r_{n-1}+1$ for all $n\in\mathbb{N}$.
\end{theorem}

\begin{proof}
\emph{Existence.} It follows from the proof of Lemma \ref{lemmacylalternatingPerron} that for each $x\in(0,1)\backslash IS^{P^-}$, the number $q_n$ in the base of a $P^-$-cylinder $\Delta^{P^-}_{q_1 \ldots q_k}$ of rank $k$ satisfying inequality \eqref{magoralternatingPerrontheorem} remains constant for all $k\geq n$.

Thus, the sequence $P$ and the real number $x\in(0,1)\backslash IS^{P^-}$ generate the sequence of natural numbers $(q_n)_{n=1}^{\infty}$. We will show that $\Delta^{P^-}_{q_1 q_2 \ldots}=x$.

For each $k\in\mathbb{N}$, we have $\Delta^{P^-}_{q_1 q_2 \ldots}\in\Delta^{P^-}_{q_1 \ldots q_k}$, and therefore $\inf\Delta^{P^-}_{q_1 \ldots q_k}<\Delta^{P^-}_{q_1 q_2 \ldots}<\sup\Delta^{P^-}_{q_1 \ldots q_k}$. Also $\inf\Delta^{P^-}_{q_1 \ldots q_k}<x<\sup\Delta^{P^-}_{q_1 \ldots q_k}$, Thus, for each $k\in\mathbb{N}$,
$$0<|x-\Delta^{P^-}_{q_1 q_2 \ldots}|<\sup\Delta^{P^-}_{q_1 \ldots q_k}-\inf\Delta^{P^-}_{q_1 \ldots q_k}=|\Delta^{P^-}_{q_1 \ldots q_k}|.$$
Since $|\Delta^{P^-}_{q_1\ldots q_k}|\to 0\text{ as }k\to\infty$ (see Lemma \ref{limcylalternatingPerron}), it follows that $\Delta^{P^-}_{q_1 q_2 \ldots}=x$.
	
\emph{Uniqueness.} Assume that, for some sequence of functions $P$ and some $x$, there exist two different $P^-$-representations, i.e.,
	$$x=\Delta^{P^-}_{q_1 q_2 \ldots}=\Delta^{P^-}_{q'_1 q'_2 \ldots}.$$
Let $k$ be a natural number such that $q_i=q'_i$ for all $i<k$ and $q_k\neq q'_k$. Then $x\in\Delta^{P^-}_{q_1 \ldots q_{k-1} q_k}$ and $x\in\Delta^{P^-}_{q_1 \ldots q_{k-1} q'_k}$, so
	$$\left\{ 
	\begin{aligned} 
		&\inf\Delta^{P^-}_{q_1 \ldots q_{k-1} q_k}<x<\sup\Delta^{P^-}_{q_1 \ldots q_{k-1} q_k};\\
		&\inf\Delta^{P^-}_{q_1 \ldots q_{k-1} q'_k}<x<\sup\Delta^{P^-}_{q_1 \ldots q_{k-1} q'_k}.
	\end{aligned} 
	\right.$$
But this contradicts Lemma \ref{lemmacylalternatingPerron}. Therefore, the assumption made above is false. Hence, for each sequence of functions $P$ every number $x\in(0,1)\setminus IS^{P^-}$ has a unique $P^-$-representation.
\end{proof}

\begin{corollary}\label{corollarystructurealternatingcyl}
A $P^-$-cylinder $\Delta^{P^-}_{q_1 \ldots q_n}$ is a set of the type $(a,b)\setminus IS^{P^-}$; the Lebesgue measure of this $P^-$-cylinder is equal to $|\Delta^{P^-}_{q_1 \ldots q_n}|$.
\end{corollary}

\begin{corollary}
	If $x=\Delta^{P^-}_{q_1 q_2 \ldots}$, then $\displaystyle\lim_{n\to\infty}\inf\Delta^{P^-}_{q_1 \ldots q_n}=\lim_{n\to\infty}\sup\Delta^{P^-}_{q_1 \ldots q_n}=x$.
\end{corollary}

\begin{corollary}
	Let $\Delta^{P^-}_{b_1 \ldots b_n}$ and $\Delta^{P^-}_{c_1 \ldots c_m}$ be two distinct $P^-$-cylinders. Then
	\begin{enumerate}[label=\upshape(\roman*), leftmargin=*, widest=ii]
		\item $\Delta^{P^-}_{b_1 \ldots b_n}\cap\Delta^{P^-}_{c_1 \ldots c_m}=\varnothing$ if there exists a natural number $k\leq\min\{n,m\}$ such that $b_i=c_i$ for all $i<k$ and $b_k\neq c_k$;
		\item $\Delta^{P^-}_{b_1 \ldots b_n}\subset\Delta^{P^-}_{c_1 \ldots c_m}$ if $n>m$ and $b_i=c_i$ for all $i\leq m$.
	\end{enumerate}
\end{corollary}

\begin{corollary}
	\begin{align*}
		&(0,1)\setminus IS^{P^-}=\bigcup_{i=r_0+1}^{\infty}\Delta^{P^-}_{i}, &&\sum_{i=r_0+1}^{\infty}|\Delta^{P^-}_{i}|=1,\\
		&\Delta^{P^-}_{q_1 \ldots q_n}=\bigcup_{i=r_n+1}^{\infty}\Delta^{P^-}_{q_1 \ldots q_n i}, &&|\Delta^{P^-}_{q_1 \ldots q_n}| =\sum_{i=r_n+1}^{\infty}|\Delta^{P^-}_{q_1 \ldots q_n i}|,
	\end{align*}
	where $r_0=\varphi_0$ and $r_n=\varphi_n(q_1,\ldots,q_n)$ for all $n\in\mathbb{N}$.
\end{corollary}

\begin{corollary}
	\textbf{Basic metric ratio}. For each $i\geq r_n+1$,
	$$\frac{|\Delta^{P^-}_{q_1 \ldots q_n i}|}{|\Delta^{P^-}_{q_1 \ldots q_n}|}=\frac{r_n}{i(i+1)},$$
	where $r_0=\varphi_0$ and $r_n=\varphi_n(q_1,\ldots,q_n)$ for all $n\in\mathbb{N}$.
\end{corollary}

\begin{theorem}\label{comparingalternating}
	Let $x=\Delta^{P^-}_{q_1 q_2 \ldots}$ and $x'=\Delta^{P^-}_{q'_1 q'_2 \ldots}$, where $k\in\mathbb{N}$, $q_i=q'_i$ for all $i<k$, and $q_k<q'_k$. Then
	\begin{enumerate}[label=\upshape(\roman*), leftmargin=*, widest=ii]
		\item $x<x'$ if $k$ is even,
		\item $x>x'$ if $k$ is odd.
	\end{enumerate}
\end{theorem}

\begin{proof}
\emph{1. Let $k$ be an even.} Then $x\in\Delta^{P^-}_{q_1 \ldots q_{k-1}q_k}$ and $x'\in\Delta^{P^-}_{q_1 \ldots q_{k-1}q'_k}$. It follows from \eqref{inf1alternatingPerron3} and \eqref{sup1alternatingPerron3} that	
	\begin{align*}
		x&<\sup\Delta^{P^-}_{q_1 \ldots q_{k-1}q_k}=\sup \Delta^{P^-}_{q_1 \ldots q_{k-1}}-\frac{r_0 r_1\cdots r_{k-1}}{(q_1-1)q_1\cdots (q_{k-1}-1)q_{k-1}q_{k}}\\
		&\leq\sup \Delta^{P^-}_{q_1 \ldots q_{k-1}}-\frac{r_0 r_1\cdots r_{k-1}}{(q_1-1)q_1\cdots (q_{k-1}-1)q_{k-1}(q'_{k}-1)}\\
		&=\inf\Delta^{P^-}_{q_1 \ldots q_{k-1}q'_k}<x',
	\end{align*}
where $r_0=\varphi_0$ and $r_n=\varphi_n(q_1,\ldots,q_n)$ for all $n=1,\ldots,k-1$.
	
\emph{2. Let $k$ be an odd.} If $k=1$, then 
	$$x>\inf\Delta^{P^-}_{q_1}=\frac{r_0}{q_1}\geq\frac{r_0}{q'_1-1}=\sup\Delta^{P^-}_{q'_1}>x'.$$
If $k\geq 3$, it follows from \eqref{sup2alternatingPerron3} and \eqref{inf2alternatingPerron3} that
	\begin{align*}
		x&>\inf\Delta^{P^-}_{q_1 \ldots q_{k-1}q_k}=\inf \Delta^{P^-}_{q_1 \ldots q_{k-1}}+\frac{r_0 r_1\cdots r_{k-1}}{(q_1-1)q_1\cdots (q_{k-1}-1)q_{k-1} q_{k}}\\
		&\geq\inf \Delta^{P^-}_{q_1 \ldots q_{k-1}}+\frac{r_0 r_1\cdots r_{k-1}}{(q_1-1)q_1\cdots (q_{k-1}-1)q_{k-1}(q'_{k}-1)}\\
		&=\sup\Delta^{P^-}_{q_1 \ldots q_{k-1}q'_k}>x',
	\end{align*}
	where $r_0=\varphi_0$ and $r_n=\varphi_n(q_1,\ldots,q_n)$ for all $n=1,\ldots,k-1$.
\end{proof}

It follows from Lemma \ref{lemmaalternatingPerron1} and Theorem \ref{mainalternatingtheorem} that there exists a bijection from $(0,1)\setminus IS^{P^-}$ to the set of all sequences $(q_n)^{\infty}_{n=0}$ such that $q_n\in\mathbb{N}$ and $q_n\geq\varphi_{n-1}(q_1,\ldots,q_{n-1})+1$ for all $n\in\mathbb{N}\cup\{0\}$.

It directly follows from the proof of Theorem \ref{lemmacylalternatingPerron} that the following theorem holds.

\begin{theorem}
Let the $P^-$-representation be generated by the sequence $P$ of functions $\varphi_n$. For each $x\in(0,1)\setminus IS^{P^-}$, $P^-$-digits of $x$ are calculated using the following recursive formulas:
	\begin{align*}
		&	r_0=\varphi_0,~~~q_1(x)=\left\lfloor\frac{r_0}{x}\right\rfloor+1,\\
		&	x_k=\left|x-\sum_{n=0}^{k-1}\frac{(-1)^n\cdot r_0\cdots r_n}{(q_1-1)q_1\cdots (q_n-1)q_n (q_{n+1}-1)}\right|,\\
		&	r_k=\varphi_k(q_1,\ldots,q_k),\\
		&	q_{k+1}(x)=\left\lfloor\frac{r_0 \cdots r_k }{(q_1-1)q_1\cdots (q_k-1)q_k x_k}\right\rfloor+1,
	\end{align*}
where $\lfloor x\rfloor$ denotes the integer part of $x$.
\end{theorem}

Relative positions of $P^-$-cylinders are illustrated by Figures \ref{fig:3}, \ref{fig:4}, and \ref{fig:5}, which  give an idea of the geometry of the $P^-$-representation.

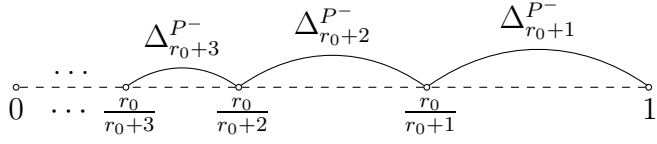
\begin{figure}[h]
	\begin{center}
		\begin{tikzpicture}[out=145,in=35]
			\draw[dashed, {Circle[open,scale=.7]}-] (8.5,0) -- (5.55,0)
			node[pos=0.01,below] {1} node[pos=1,below] {$\frac{r_0}{r_0+1}$};
			\draw[shorten <=1.2pt,shorten >=1.2pt] (8.47,0) to node[above] {$\Delta^{P^-}_{r_0+1}$} (5.5,0);
			\draw[dashed, {Circle[open,scale=.7]}-] (5.55,0) -- (3.05,0)
			node[pos=1,below] {$\frac{r_0}{r_0+2}$};
			\draw[shorten <=1.2pt,shorten >=1.2pt] (5.5,0) to node[above] {$\Delta^{P^-}_{r_0+2}$} (3,0);
			\draw[dashed, {Circle[open,scale=.7]}-] (3.05,0) -- (1.55,0)
			node[pos=1,below] {$\frac{r_0}{r_0+3}$};
			\draw[shorten <=1.2pt,shorten >=1.2pt] (3,0) to node[above] {$\Delta^{P^-}_{r_0+3}$} (1.5,0);
			\draw[dashed, {Circle[open,scale=.7]}-{Circle[open,scale=.7]}] (1.55,0) -- node[above=1] {$\ldots$} node[below=2.5] {$\cdots$} (0,0)
			node[pos=0.97,below] {0};
		\end{tikzpicture}
	\end{center}
	\caption{\small $P^-$-cylinders of rank $1$ inside $(0,1]$, $r_0=\varphi_0$.}
	\label{fig:3}
\end{figure}

\begin{figure}[h]
	\begin{center}
		\begin{tikzpicture}[out=145,in=35]
			\draw[dashed, {Circle[open,scale=.7]}-] (8.5,0) -- (5.55,0)
			node[pos=0,below] {$\sup\Delta^{P^-}_{c_1\ldots c_k}$};
			\draw[shorten <=1.2pt,shorten >=1.2pt] (8.47,0) to node[above] {$\Delta^{P^-}_{c_1\ldots c_k [r_k+1]}$} (5.5,0);
			\draw[dashed, {Circle[open,scale=.7]}-] (5.55,0) -- (3.05,0);
			\draw[shorten <=1.2pt,shorten >=1.2pt] (5.5,0) to node[above] {$\Delta^{P^-}_{c_1\ldots c_k [r_k+2]}$} (3,0);
			\draw[dashed, {Circle[open,scale=.7]}-] (3.05,0) -- (1.55,0);
			\draw[shorten <=1.2pt,shorten >=1.2pt] (3,0) to node[above] {$\Delta^{P^-}_{c_1\ldots c_k [r_k+3]}$} (1.5,0);
			\draw[dashed, {Circle[open,scale=.7]}-{Circle[open,scale=.7]}] (1.55,0) -- node[above=1] {$\ldots$} (0,0)
			node[pos=1,below] {$\inf\Delta^{P^-}_{c_1\ldots c_k}$};
			\draw[shorten <=2.5pt,shorten >=2.5pt][out=-10,in=-170] (0,0) to node[below] {$\Delta^{P^-}_{c_1\ldots c_k}$} (8.5,0);
		\end{tikzpicture}
	\end{center}
	\caption{\small $P^-$-cylinders of \textbf{odd} rank $k+1$ inside the $P^-$-cylinder $\Delta^{P^-}_{c_1\ldots c_k}$ \\ of \textbf{even} rank $k$, $r_k=\varphi_k(c_1,\ldots,c_k)$.}
	\label{fig:4}
\end{figure}
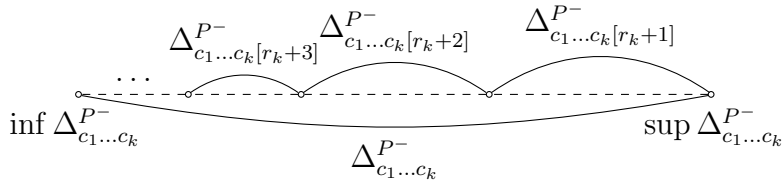

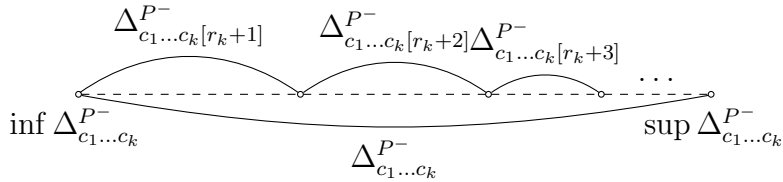
\begin{figure}[h]
	\begin{center}
		\begin{tikzpicture}[out=35,in=145]
			\draw[dashed, {Circle[open,scale=.7]}-] (0,0) -- (2.95,0)
			node[pos=0,below] {$\inf\Delta^{P^-}_{c_1\ldots c_k}$};
			\draw[shorten <=1.2pt,shorten >=1.2pt] (0.03,0) to node[above] {$\Delta^{P^-}_{c_1\ldots c_k [r_k+1]}$} (3,0);
			\draw[dashed, {Circle[open,scale=.7]}-] (2.95,0) -- (5.45,0);
			\draw[shorten <=1.2pt,shorten >=1.2pt] (3,0) to node[above] {$\Delta^{P^-}_{c_1\ldots c_k [r_k+2]}$} (5.5,0);
			\draw[dashed, {Circle[open,scale=.7]}-] (5.45,0) -- (6.95,0);
			\draw[shorten <=1.2pt,shorten >=1.2pt] (5.5,0) to node[above] {$\Delta^{P^-}_{c_1\ldots c_k [r_k+3]}$} (7,0);
			\draw[dashed, {Circle[open,scale=.7]}-{Circle[open,scale=.7]}] (6.95,0) -- node[above=1] {$\ldots$} (8.5,0)
			node[pos=1,below] {$\sup\Delta^{P^-}_{c_1\ldots c_k}$};
			\draw[shorten <=2.5pt,shorten >=2.5pt][out=-10,in=-170] (0,0) to node[below] {$\Delta^{P^-}_{c_1\ldots c_k}$} (8.5,0);
		\end{tikzpicture}
	\end{center}
	\caption{\small $P^-$-cylinders of \textbf{even} rank $k+1$ inside the $P^-$-cylinder $\Delta^{P^-}_{c_1\ldots c_k}$ \\ of \textbf{odd} rank $k$, $r_k=\varphi_k(c_1,\ldots,c_k)$.}
	\label{fig:5}
\end{figure}

\section{Faithfulness of families of $P$-cylinders and $P^-$-cylinders for calculating the Lebesgue measure of a numerical set}

Let $\Phi$ be a family of intervals that are subsets of interval $\langle a,b\rangle\subseteq\mathbb{R}$ (open, close, or half-open). If there exist countable covers of sets $E\subset\langle a,b\rangle$ and $\langle a,b\rangle\setminus E$ by intervals from $\Phi$:
\begin{gather*}
	\bigcup_{n} U_n\supset E,~~~\bigcup_{n} V_n\supset \langle a,b\rangle\setminus E,~~~U_n\in\Phi,~~~V_n\in\Phi,
\end{gather*}
then we define the following quantities for $E$:
\begin{gather*}
	\lambda^*(E,\Phi)=\inf\left\{\sum_n \left|U_n\right|\right\},~~~\lambda_*(E,\Phi)=(b-a)-\inf\left\{\sum_n \left|V_n\right|\right\}.
\end{gather*}

We say that the set $E$ is \emph{$\Phi$-measurable} if $\lambda^*(E,\Phi)=\lambda_*(E,\Phi)$. In this case, we denote $\lambda(E,\Phi)=\lambda^*(E,\Phi)=\lambda_*(E,\Phi)$.

\begin{definition}
The family $\Phi$ of subintervals of the interval $\langle a,b\rangle$ is called a \emph{faithful} family for calculating the Lebesgue measure $\lambda$ of subsets of $\langle a,b\rangle$ if every set $E\subset\langle a,b\rangle$ is $\Phi$-measurable if and only if $E$ is Lebesgue-measurable, and $\lambda(E,\Phi)=\lambda(E).$
\end{definition}

If $\Phi$ is the family of all open intervals, then $\Phi$-measurability is equivalent to Lebesgue measurability, and $\lambda(\cdot,\Phi)\equiv\lambda(\cdot)$.

\begin{lemma}\label{lemmadovirchistlebesgue}
For each open interval $(x_1,x_2)\subset(0,1]$, each $P$-representa\-tion, and each $\varepsilon>0$, there exists a countable cover $\bigcup\Delta^P_{p_1\ldots p_n}$ of $(x_1,x_2)$ by $P$-cylinders such that
	$$\sum \left|\Delta^P_{p_1\ldots p_n}\right|<x_2-x_1+\varepsilon.$$
\end{lemma}

\begin{proof}
\emph{1. Let $x_1>0$, $x_1=\Delta^P_{c_1 \ldots c_{k-1} a_k a_{k+1}\ldots}$, $x_2=\Delta^P_{c_1 \ldots c_{k-1} b_k b_{k+1}\ldots}$, and moreover $a_k>b_k$.} Let us fix some $\varepsilon>0$. Then, according to {\cite[Lemma 4]{Moroz2024}}, there exists a number $t>k$ such that
	$$\left|\Delta^P_{c_1 \ldots c_{k-1} a_k \ldots a_t}\right|<\frac{\varepsilon}{2},~~~\left|\Delta^P_{c_1 \ldots c_{k-1} b_k \ldots b_t}\right|<\frac{\varepsilon}{2}.$$
Therefore,
	\begin{align*}
		\inf\Delta^P_{c_1 \ldots c_{k-1} a_k \ldots a_t}<x_1\leq&\sup\Delta^P_{c_1 \ldots c_{k-1} a_k \ldots a_t}\\
		\leq&\inf\Delta^P_{c_1 \ldots c_{k-1} b_k \ldots b_t}<x_2\leq\sup\Delta^P_{c_1 \ldots c_{k-1} b_k \ldots b_t}.
	\end{align*} 
We have the cover of $(x_1,x_2)$ by three sets:
	\begin{gather*}\label{7}
		\Delta^P_{c_1 \ldots c_{k-1} a_k \ldots a_t},~\Delta^P_{c_1 \ldots c_{k-1} b_k \ldots b_t},~\left(\sup\Delta^P_{c_1 \ldots c_{k-1} a_k \ldots a_t},\inf\Delta^P_{c_1 \ldots c_{k-1} b_k \ldots b_t}\right].
	\end{gather*}
The half-open interval $\left(\sup\Delta^P_{c_1 \ldots c_{k-1} a_k \ldots a_t},\inf\Delta^P_{c_1 \ldots c_{k-1} b_k \ldots b_t}\right]$ can be represented as a countable union of pairwise disjoint $P$-cylinders. Together with $P$-cylinders $\Delta^P_{c_1 \ldots c_{k-1} a_k \ldots a_t}$ and $\Delta^P_{c_1 \ldots c_{k-1} b_k \ldots b_t}$, this union covers $(x_1,x_2)$ by $P$-cylinders, whose total length is $\sup\Delta^P_{c_1 \ldots c_{k-1} b_k \ldots b_t}-\inf\Delta^P_{c_1 \ldots c_{k-1} a_k \ldots a_t}$; here 
	\begin{align*}
		&\sup\Delta^P_{c_1 \ldots c_{k-1} b_k \ldots b_t}-\inf\Delta^P_{c_1 \ldots c_{k-1} a_k \ldots a_t}\\
		=~&x_2-x_1+\left(x_1-\inf\Delta^P_{c_1 \ldots c_{k-1} a_k \ldots a_t}\right)+\left(\sup\Delta^P_{c_1 \ldots c_{k-1} b_k \ldots b_t}-x_2\right)\\
		\leq~&x_2-x_1+\left|\Delta^P_{c_1 \ldots c_{k-1} a_k \ldots a_t}\right|+\left|\Delta^P_{c_1 \ldots c_{k-1} b_k \ldots b_t}\right|<x_2-x_1+\varepsilon.
	\end{align*}
	
\emph{2. Let $x_1=0$, $x_2=\Delta^P_{c_1 c_2 \ldots}$.}
Let us fix some $\varepsilon>0$. Then, according to {\cite[Lemma 4]{Moroz2024}}, there exists a number $t$ such that $\left|\Delta^P_{c_1 \ldots c_{t}}\right|<\varepsilon.$
Of course,
	$$0=x_1<\sup\Delta^P_{c_1 \ldots c_{t-1}[c_{t}+1]}=\inf\Delta^P_{c_1 \ldots c_{t}}<x_2\leq\sup\Delta^P_{c_1 \ldots c_{t}}.$$ 
We have the cover of $(x_1,x_2)$ by two sets:
	$$\Delta^P_{c_1 \ldots c_{t}},~\left(0,\sup\Delta^P_{c_1 \ldots c_{t-1}[c_{t}+1]}\right].$$
Similarly to case 1, this cover of $(x_1,x_2)$ can be represented as a countable union of pairwise disjoint $P$-cylinders. The total length of $P$-cylinders in this union is $\sup\Delta^P_{c_1 \ldots c_{t}}$, satisfying the inequality
	$$\sup\Delta^P_{c_1 \ldots c_{t}}=x_2+\left(\sup\Delta^P_{c_1 \ldots c_{t}}-x_2\right)\leq x_2+\left|\Delta^P_{c_1 \ldots c_{t}}\right|<x_2-x_1+\varepsilon.$$
\end{proof}

Let $\mathfrak{P_0}$ be the family of all $P$-cylinders of some $P$-representation.

\begin{theorem}\label{theoremdovirchistlebesgue}
For each $P$-representation, the family $\mathfrak{P_0}$ of $P$-cylinders is faithful family for calculating the Lebesgue measure $\lambda$ of subsets of $(0,1]$.
\end{theorem}

\begin{proof}
Let $E\subset(0,1]$ and let $\displaystyle\bigcup_n A_n$ be a countable cover of $E$ by open subintervals of $(0,1]$. According to Lemma \ref{lemmadovirchistlebesgue}, for all $\varepsilon>0$, each open interval $A_n$ can be covered by a countable set of $P$-cylinders, whose total length is at most $\left|A_n\right|+\frac{\varepsilon}{2^n}$. Thus, we obtain a countable cover of $E$ by $P$-cylinders $\displaystyle\bigcup\Delta^P_{c_1\ldots c_k}$, whose total length is at most $\displaystyle\sum_{n}\left|A_n\right|+\varepsilon$. Therefore,
	$$\inf\left\{\sum\left|\Delta^P_{c_1\ldots c_k}\right|\right\}=\inf\left\{\sum_{n}|A_n|\right\},$$
where infima are over all possible countable covers of $E$ by $P$-cylinders and open subintervals of $(0,1]$, respectively. Hence $\lambda^*(E,\mathfrak{P_0})=\lambda^*(E)$, where $\lambda^*$ is the Lebesgue outer measure.
	
Similarly, $\lambda_*(E,\mathfrak{P_0})=\lambda_*(E)$, where $\lambda_*$ is the Lebesgue inner measure.
	
Thus, $\lambda^*(E,\mathfrak{P_0})=\lambda_*(E,\mathfrak{P_0})$ if and only if $\lambda^*(E)=\lambda_*(E)$. Therefore, for each $P$-representation, the family $\mathfrak{P_0}$ of $P$-cylinders is faithful family for calculating the Lebesgue measure of subsets of $(0,1]$.
\end{proof}

The following theorem is analogous to Theorem \ref{theoremdovirchistlebesgue}. Therefore, we state it here without proof.

\begin{theorem}\label{theoremdovirchistlebesguealt}
For each $P^-$-representation, the family $\mathfrak{P^-_0}$ of all $P^-$-cylin\-ders is faithful family for calculating the Lebesgue measure of subsets of $(0,1)\setminus IS^{P^-}$.
\end{theorem}

\section{The equivalence between the metric theories of $P$-representation and $P^-$-representation} 

Let the $P$-representation and the $P^-$-representation be generated by the sequence $P$ of functions $\varphi_n$. We consider function $F_P\colon (0,1]\rightarrow(0,1)\setminus IS^{P^-}$ such that
$$F_P(x)=F_P\left(\Delta^P_{p_1 p_2 \ldots}\right)=\Delta^{P^-}_{p_1 p_2 \ldots}$$ 
for all $x=\Delta^P_{p_1 p_2 \ldots}\in(0,1]$.

The following properties follow directly from the definition of the function $F_P$:
\begin{enumerate}[label=\upshape(\roman*), leftmargin=*, widest=ii]
	\item the function $F_P$ is a bijection between the sets $(0,1]$ and $(0,1)\setminus IS^{P^-}$;
	\item the function $F_P$ transforms each $P$-cylinder $\Delta^{P}_{c_1 \ldots c_n}$ into a $P^-$-cylinder $\Delta^{P^-}_{c_1 \ldots c_n}$ with the same diameter (Lebesgue measure).
\end{enumerate} 

\begin{theorem}\label{maintheorem}
If $E\subseteq(0,1]$ is Lebesgue-measurable, then its image $F_P(E)$ under $F_P$ is also Lebesgue-measurable, and $\lambda(F_P(E))=\lambda(E)$.
\end{theorem}

\begin{proof}
Let $E\subseteq(0,1]$ be a Lebesgue-measurable set, and let $\overline{E}=(0,1]\setminus E$. It follows from Theorem \ref{theoremdovirchistlebesgue} that 
	\begin{gather*}
		\lambda(E,\mathfrak{P_0})=\lambda^*(E,\mathfrak{P_0})=\lambda_*(E,\mathfrak{P_0})=\lambda(E),\\
		\lambda(\overline{E},\mathfrak{P_0})=\lambda^*(\overline{E},\mathfrak{P_0})=\lambda_*(\overline{E},\mathfrak{P_0})=\lambda(\overline{E}).
	\end{gather*}
The image under $F_P$ of the cover $\bigcup \Delta^{P}_{a_1\ldots a_n}$ of $E$ is the cover $\bigcup \Delta^{P^-}_{a_1\ldots a_n}$ of $F_P(E)$, and similarly, the image under $F_P$ of the cover $\bigcup\Delta^{P}_{b_1\ldots b_n}$ of $\overline{E}$ is the cover $\bigcup \Delta^{P^-}_{b_1\ldots b_n}$ of $F_P(\overline{E})$. The corresponding covers (image and preimage) have the same total length of cylinders. Thus,
	\begin{align*}
		\lambda^*(F_P(E),\mathfrak{P^-_0})&=\inf\left\{\sum\left|\Delta^{P^-}_{a_1\ldots a_n}\right|\right\}=\inf\left\{\sum\left|\Delta^{P}_{a_1\ldots a_n}\right|\right\}\\
		&=\lambda^*(E,\mathfrak{P_0})=\lambda(E).
	\end{align*}
Similarly, $\lambda^*(F_P(\overline{E}),\mathfrak{P^-_0})=\lambda(\overline{E})$. Since $F_P(\overline{E})=\left((0,1)\setminus IS^{P^-}\right)\setminus F_P(E)$, it follows that
	\begin{align*}
		\lambda_*(F_P(E),\mathfrak{P^-_0})=1-\lambda^*(F_P(\overline{E}),\mathfrak{P^-_0})=1-\lambda(\overline{E})=\lambda(E).
	\end{align*}
We have
	\begin{gather}\label{8}
		\lambda^*(F_P(E),\mathfrak{P^-_0})=\lambda_*(F_P(E),\mathfrak{P^-_0})=\lambda(E).
	\end{gather}
It follows that $F_P(E)$ is $\mathfrak{P^-_0}$-measurable. Using \eqref{8} and Theorem \ref{theoremdovirchistlebesguealt}, we obtain that $F_P(E)$ is Lebesgue-measurable, and simultaneously $\lambda(F_P(E))=\lambda^*(F_P(E),\mathfrak{P^-_0})=\lambda(E)$.
\end{proof}

\begin{remark}
Here, we did not use the theory of measure-preserving transformations, since obtaining our results using it would not be easy.
\end{remark}

\begin{corollary}\label{maincorollary}
Let the $P$-representation and the $P^-$-representation be generated by the sequence $P$. Suppose that $E_1\subset(0,1]$, $E_2\subset(0,1)\setminus IS^{P^-}$, and at least one of these sets, $E_1$ or $E_2$, is Lebesgue-measurable. Moreover, suppose that the set of sequences of $P$-digits corresponding to numbers in $E_1$ is equal to the set of sequences of $P^-$-digits corresponding to numbers in $E_2$. Then $\lambda(E_1)=\lambda(E_2)$.
\end{corollary}

It follows from the fact that $E_2=F_P(E_1)$.

From Theorem \ref{maintheorem} and Corollary \ref{maincorollary}, we can conclude that calculating the Lebesgue measure of a set defined in terms of the $P^-$-representation reduces to calculating the Lebesgue measure of the corresponding set defined in terms of $P$-representation.

\section{Known analogies that can be explained by Theorem \ref{maintheorem}}

Next, we provide examples where Theorem \ref{maintheorem} and Corollary \ref{maincorollary} explain existing analogies in individual theories of representing real numbers by series which are special cases of positive and alternating Perron series. Note that series \eqref{alternatingPerronSer1} defines to the notation of the alternating L\"{u}roth series, the Ostrogradsky--Sierpi\'nski--Pierce series, and the second Ostrogradsky series (the alternating Sylvester series) that are slightly different from their traditional forms. Specifically, the $P$-digits of numbers are incremented by 1 compared to the corresponding elements in the traditional notation of these series. However, many properties, including statistical and asymptotic properties, are common to both type of notation or can be easily adapted to them. Properties of $P^-$-representations mentioned later will be adapted to the notational form \eqref{alternatingPerronSer1}.

\emph{1. Analogies between the modified Engel expansion and the Ostrogradsky--Sierpi\'nski--Pierce expansion.}

Let the $P$-representation and the $P^-$-representation be generated by the sequence $P$ of functions $\varphi_n$ such that $\varphi_0=1$, $\varphi_n(x_1,\ldots,x_n)=x_n$ for all $n\in\mathbb{N}$. Then the $P$-representation is the representation of numbers by the modified Engel series, and the $P^-$-representation is the representation of numbers by the Ostrogradsky--Sierpi\'nski--Pierce series. 

In \cite{Renyi1962-2} (see also \cite{Renyi1962-1}), R\'enyi proved the following properties:
\begin{gather*}
	\lim_{n\to\infty}\sqrt[n]{p_n(x)}=e,\\
	\limsup_{n\to\infty}\frac{\log p_n(x)-n}{\sqrt{2n\log\log n}}=1,\\
	\liminf_{n\to\infty}\frac{\log p_n(x)-n}{\sqrt{2n\log\log n}}=-1
\end{gather*}
for almost all $x\in(0,1]$.
He also proved that
\begin{gather*}
	\lim_{n\to\infty}\mathbf{P}\left(\frac{\log p_n(x)-n}{\sqrt{n}}<y\right)=\frac{1}{\sqrt{2\pi}}\int\limits_{-\infty}^{y} e^{-\frac{u^2}{2}}\mathrm{d}u,
\end{gather*}
where $p_n(x)$ is the $n$th element of the expansion of $x$ by the modified Engel series (see {\cite[Theorem 1--3]{Renyi1962-2}}). More than 20 years later, Shallit proved similar results for the elements of the Ostrogradsky--Sierpi\'nski--Pierce expansion (see {\cite[Theorem 16--18]{Shallit1986}}). However, it follows from Theorem \ref{maintheorem} and Corollary \ref{maincorollary} that the results of Shallit are corollaries of the results of R\'enyi and do not require individual proofs.

\emph{2. Analogies between the Sylvester expansion and the alternating Sylves\-ter expansion (the second Ostrogradsky expansion).}

Analogous properties for difference forms of positive and alternating Sylvester series (see {\cite[Corollary 4]{Moroz2023}} and {\cite[Theorem 3.1]{TorbinPratsyovyta2010}}) are also explained by Theorem \ref{maintheorem} and Corollary \ref{maincorollary}.

\emph{3. Analogies between the L\"{u}roth expansion and the alternating L\"{u}roth expansion.}

Let the $P$-representation and the $P^-$-representation be generated by the sequence $P$ of functions $\varphi_n$ such that $\varphi_n=1$ for all $n\in\mathbb{N}\cup\left\{0\right\}$. Then the $P$-representation is the representation of numbers by the L\"{u}roth series, and the $P^-$-representation is the representation of numbers by the alternating L\"{u}roth series.

In \cite{ZhPr2012}, Zhykharyeva and Pratsiovytyi investigated the topological, metric, and fractal properties of the set
$$C=\left\{x\colon x=\Delta^{P}_{p_1 p_2 \ldots},~p_n\in V\subset\mathbb{N},~\forall n\in\mathbb{N}\right\}.$$
In particular, it is proved that if $V\neq\mathbb{N}$, then $C$ is a null set (see {\cite[Theorem 4, item 2]{ZhPr2012}}). In \cite{PK2013}, Pratsiovytyi and Khvorostina investigated the similar set
$$C^-=\left\{x\colon x=\Delta^{P^-}_{q_1 q_2 \ldots},~q_n\in V\subset\mathbb{N},~\forall n\in\mathbb{N}\right\}.$$
Expectedly, $C^-$ has the same property (see {\cite[Corollary 1.3, item 2]{PK2013}}).

In \cite{Salat1968}, \v{S}al\'at proved the existence of the Khintchine-type constant (see \cite[Satz 2.3]{Salat1968}) and calculated the asymptotic frequencies of digits (see {\cite[Satz 2.5]{Salat1968}}) for the L\"{u}roth expansion. In \cite{KalpazidouKnopfmacher1991}, S.~Kalpazidou, A.~Knopfmacher, and J.~Knopfmacher proved analogous properties of the alternating L\"{u}roth expansion (see {\cite[Theorem 1]{KalpazidouKnopfmacher1991}}).

\emph{4. Difference between the Engel expansion and the Ostrograd\-sky--Sierpi\'n\-ski--Pierce expansion.}

In \cite{Zhu2014}, Zhu studied statistical properties of elements of Engel expansion. In \cite{Fang2015}, Fang studied Ostrogradsky--Sierpi\'nski--Pierce expansion and solved problems analogous to those in \cite{Zhu2014}. Note that Fang viewed the Ostrogradsky--Sierpi\'nski--Pierce expansion as the alternating analogue of the Engel expansion, referring to it as the alternating Engel expansion. But the results obtained in \cite{Fang2015} significantly differ from those in \cite{Zhu2014} (see {\cite[Remark 1]{Fang2015}}). However, these expansions are defined as $P$-representation and $P^-$-representation using different sequences $P$ of functions $\varphi_n$. This is the reason for the difference in results, rather than whether the expansions are positive or alternating.

In \cite{A2024}, Ahn also considered the Ostrogradsky--Sierpi\'nski--Pierce expansion as the alternating Engel expansion. In our opinion, it is more appropriate to call the Ostrogradsky--Sierpi\'nski--Pierce expansion the alternating modified Engel expansion.

\subsection*{Acknowledgements}
This work was supported by a grant from the Simons Foundation (1290607, M.M.).

\end{document}